\documentclass[12pt, reqno,final]{amsart}
\usepackage[left=2.8cm,right=2.8cm,top=2.8cm,bottom=2.8cm]{geometry}
\usepackage[utf8]{inputenc}
\usepackage{amsmath,amssymb,amsfonts,amsthm,mathrsfs,enumerate,cite,graphicx}
\usepackage[dvipsnames]{xcolor}
\usepackage{enumerate}
\usepackage{comment}
\usepackage{esint} 
\usepackage{showkeys}
\usepackage{hyperref}
\usepackage{autonum}
\newtheorem{theorem}{Theorem} 
\newtheorem{lemma}[theorem]{Lemma}

\theoremstyle{definition}
 
\newtheorem{definition}[theorem]{Definition} 
\newtheorem{remark}[theorem]{Remark} 
\newcommand{\dint}{\mathrm{d}}
\newcommand{\RR}{\mathbb{R}}
\newcommand{\NN}{\mathbb{N}}

\renewcommand{\SS}{\mathbb{S}}
\newcommand{\Haus}{\mathcal{H}}
\newcommand{\Leb}{\mathcal{L}}
\newcommand{\Radon}{\mathcal{M}}
\newcommand{\llc}{\,\text{\Large{$\llcorner$}}} 
\newcommand{\loc}{\textup{loc}}
\newcommand{\veps}{\varepsilon}
\newcommand{\wto}{\rightharpoonup}
\newcommand{\adm}{\mathcal{A}}
\newcommand{\Ebulk}{E_{\textrm{bulk}}}
\newcommand{\Einterface}{E_{\textrm{int}}}
\newcommand{\spt}{\operatorname{spt}}
\newcommand{\qsharp}{q_\sharp} 
\newcommand{\charac}{1}
\newcommand{\Cof}{{\rm Cof}\,}

\newcommand{\KB}{\color{black}}
\newcommand{\UUU}{\color{black}} 
\newcommand{\MK}{\color{black}}
\newcommand{\MAK}{\color{black}}
\newcommand{\EEE}{\color{black}}
\newcommand{\FFF}{\color{black}}
\newcommand{\NNN}{\color{black}}

\usepackage[normalem]{ulem}

\DeclareMathOperator*\argmin{\text{arg\,min}}

\title[Curvature-dependent Eulerian interfaces]{Curvature-dependent Eulerian \\ interfaces in elastic solids}

\author[K. Brazda]{Katharina Brazda} 
\address[Katharina Brazda]{Faculty of Mathematics, University of Vienna, 
Oskar-Morgenstern-Platz 1, A-1090 Vienna, Austria}
\email{katharina.brazda@univie.ac.at}

\author[M. Kru\v z\'\i k]{Martin Kru\v z\'\i k}
\address[Martin Kru\v{z}\'ik]{Czech Academy of Sciences, Institute of Information Theory and Automation, Pod vod\' arenskou ve\v z\' \i\ 4, 182 00, Prague 8, Czechia.}
\email{kruzik@utia.cas.cz}

\author[F. Rupp]{Fabian Rupp} 
\address[Fabian Rupp]{Faculty of Mathematics, University of Vienna, 
Oskar-Morgenstern-Platz 1, A-1090 Vienna, Austria}
\email{fabian.rupp@univie.ac.at}

\author[U. Stefanelli] {Ulisse Stefanelli} 
\address[Ulisse Stefanelli]{Faculty of Mathematics, University of Vienna, 
Oskar-Morgenstern-Platz 1, A-1090 Vienna, Austria,\,\&
Vienna Research Platform on Accelerating Photoreaction Discovery, University of Vienna, W\"ahringerstrasse 17, A-1090 Vienna, Austria,\,\&
Istituto di Matematica Applicata e Tecnologie Informatiche E. Magenes,
via Ferrata 1, I-27100 Pavia, Italy.
}
\email{ulisse.stefanelli@univie.ac.at}

\date{\today}

\begin{document}

 \begin{abstract}
 \MK
We propose a sharp-interface model for a hyperelastic material consisting of two phases. In this model, phase interfaces are treated in the deformed configuration, resulting in a fully Eulerian interfacial energy. \FFF In order to penalize large curvature of the interface, we include a geometric term featuring a curvature varifold. \EEE 
Equilibrium solutions are \UUU proved \MK to exist via minimization. We then utilize this model in an Eulerian topology optimization problem that incorporates \UUU a \MK curvature penalization.
\MK
 
 \end{abstract}
\maketitle

\section{Introduction}\label{sec:intro}
\MK
In the field of elasticity, it is commonly assumed that experimentally observed patterns in materials correspond to the minimization of a suitable \UUU phase-dependent \MK energy. \UUU Indeed, some \MK materials have multiple phases, and the optimal energetic configuration is \UUU often \MK achieved by creating spatial microstructures composed of these phases. These microstructures \UUU feature their 
\MK own unique size, shape, and distribution (such as grains, precipitates, dendrites, spherulites, lamellae, or pores). The phases can be distinguished from each other by their various crystalline, semicrystalline, or amorphous \UUU properties, \MK which can be \UUU experimentally identified \MK through microscopy techniques.

To fully understand the behavior of a material, it is necessary to \UUU characterize the relation \MK 
between the macroscopic properties and the underlying phenomena occurring \UUU at \MK  the microstructural scale. \UUU To shed light on the multiscale nature of this phenomenon is paramount 
\MK for optimizing material performance and developing new materials with tailored properties.

A prominent example of materials with microstructure are shape memory alloys, \UUU showing a highly symmetric crystallographic variant called austenite, preferred at high temperatures, as well as different low-symmetry variants called martensites, favored at low temperatures. \MK 
These alloys, \UUU including \MK NiTi, CuAlNi, or InTh, are widely used in various technological applications, as discussed in \cite{Jani2014ARO}. 
The mixing of these different phases lead to the formation of complex microstructures, \UUU which ultimately govern the rich thermomechanical response of the material. \MK

In the continuum theory, the \UUU total stored energy of the system  usually consists of \MK 
\UUU bulk and interfacial energy contributions. \MK 
Neglecting the interfacial energy generically leads to a minimization problem that has no solution due to the formation of spatially finer and finer oscillations of the deformation gradient among the various phases. \UUU If \MK  spatial  phase changes are penalized by  the  interfacial energy, an optimal material layout is reached by \UUU balancing \MK energy contributions rising from the bulk \UUU and \MK the interface, \UUU under the effect of \MK external loading.

\UUU Different models \MK have been considered \UUU taking \FFF into account \MK interfacial energy \UUU in various forms.  \MK This includes strain gradients  \cite{BallMora-Corral-2009,Toup62EMCS} but also  gradients of nonlinear null \UUU Lagrangians \MK of the deformation \cite{BeKrSc17NLMGP}.  \MAK Curved interfaces in solids are thoroughly studied in \cite{gurtin1998general} following previous research on interface-bulk elastic interactions \cite{gurtin1975continuum}.  Curved twin boundaries in lead orthovanadate are observed in \cite{manolikas1986local}, see also \cite{cahn1982transitions} for discussions on grain boundary shapes or \cite{gao2017curvature} for curvature-dependent interfacial energies in nanomaterials. \EEE  \UUU Recently \MK
\v{S}ilhav\'{y}  introduced  \UUU in \cite{Silhavy-2011} \MK a notion of interface polyconvexity and  proved \UUU it \MK sufficient to ensure the existence of minimizers for the corresponding static problem. \UUU In particular, it his model the 
\MK perimeters of interfaces in the reference and deformed configurations, \UUU as well as \MK the deformations of lines in the referential interfaces \UUU are penalized. \MK A more explicit characterization of interface polyconvexity can be found in \cite{GKMS19,grandi2020}, \UUU discussing the case of materials with more than two phases as well. \MK 
\UUU Again, let us mention \MK that the mathematical treatment of multiphase materials without surface-energy \UUU penalization \MK  typically leads to ill-posed problems where the existence of a solution is not necessarily guaranteed, and some relaxation 
 is needed, cf.~\cite{Daco89DMCV}. This, however, would challenge orientation preservation of the involved deformations, and consequently, also injectivity may be lost \cite{Ball_puzzles}.  

\FFF In this article, we consider a material with two phases, separated by a sharp interface. \UUU Note however that our model can \MK be extended to describe more phases   similarly as in \cite{Silhavy-2010}, \UUU see Remark \ref{rem:multi}. \FFF
To incorporate the penalization of large \UUU interface  curvatures, \FFF we describe the interface in terms of a curvature varifold, a measure-theoretic generalization of classical surfaces with a notion of curvature and \UUU with \FFF good compactness properties \cite{Hutchinson:86,Mantegazza:96,Simon:83}. \MAK Mathematical models involving varifolds \EEE have been used to describe bending-resistant interfaces in a wide range of applications, for instance in the modeling of cracks \cite{MR2658342,MR2644754,KrMaMu:22}, biological membranes \cite{EichmannAGAG,BLS:20,RS:23}, or anisotropic phase transitions \cite{Moser:12}. \MK 

\FFF The state of the elastic body is characterized by the deformation $y$ \KB of the reference configuration $\Omega\subset\RR^3$, \FFF the phase field $\phi$, and a curvature varifold $V$ describing the phase interface. The equilibrium \KB state \FFF minimizes the energy $E$, consisting of the elastic bulk energy and the energy of the phase interface. If the varifold $V$ and the phase interface correspond to a smoothly embedded surface $M \subset \RR^3$ with  second fundamental form $I\!I$, our energy typically looks as follows
\begin{align}
    E(y,\phi,V)=\int_\Omega\Big((1-\phi\circ y)W_0(\nabla y)+\phi\circ y\, W_1(\nabla y)\, \Big)\dint X \UUU + \Haus^2(M) + \int_M |I\!I|^p \,\dint \Haus^2
\end{align}
see Section \ref{subsec:energies} for \FFF the general 
\UUU definition and all necessary details. 
 \FFF

\FFF The main result of this work is \UUU the proof of \FFF 
the existence of minimizers with the phase field and the varifold defined in the deformed configuration, i.e., in the Eulerian setting. In order to find a good framework for the direct method in the Calculus of Variations, two important challenges need to be met: Firstly, a suitable coupling needs to be imposed to identify the varifold with the phase field, see Definition \ref{def:coupling}. Secondly, the Eulerian setting implies that both $\phi$ and $V$ are defined in the deformed configuration $y(\Omega)$, which itself is subject to minimization. Once compactness is achieved, the existence of minimizers follows from the closedness of the coupling condition together with the lower semicontinuity of the energy via the usual (poly-)convexity assumptions. \NNN Note that, if no curvature term was present in the energy no varifold $V$ would be needed and the existence of minimizers would follow from the theory in \cite{GKMS19}. \EEE 

\FFF Moreover, we adapt the variational theory to study a related problem in topology optimization, taking into account the curvature of the design material surface. We also provide a corresponding referential formulation which might be computationally more feasible.
\MK

This article is organized as follows. \KB In Section~\ref{sec:prelim}, \MK we introduce basic notions and notation   \KB on functions of bounded variation and varifolds. \MK Our model is presented in Section~\ref{sec:model} and  the existence of a solution is proved in Section~\ref{sec:existence}.  This allows us to settle a problem of topology optimization in the Eulerian coordinates and to establish the existence of an optimal topological design in Section~\ref{sec:topology}.

\EEE





 
\section{Notation and preliminaries}\label{sec:prelim}

\subsection{Piecewise constant functions of bounded variation}
Let $U\subset\RR^3$ be open.
By $BV(U)$ we denote the class of {\em functions of bounded variation} and by $SBV(U)$ the class of {\em special functions of bounded variation}, see \cite{AmFuPa:00}.
We set
$$
SBV(U;\{0,1\}):=\{\phi\in SBV(U):\:\phi(x)\in\{0,1\}\:\text{for a.e.}\: x\in U\}.
$$
Its elements are {\em piecewise constant functions} in the sense of \cite[Def.\ 4.21]{AmFuPa:00}, restricted to only assuming values in $\{0,1\}$.
The weak derivative of $\phi\in SBV(U;\{0,1\})$ is the $\RR^3$-valued Radon measure
$$
D\phi=\nu_\phi(\Haus^{2}\llc J_\phi)\in\Radon(U;\RR^3).
$$
Here, $\Haus^2$ is the two-dimensional Hausdorff measure, $J_\phi\subset U$ is the approximate jump set, which is countably $\Haus^2$-rectifiable, and $\nu_\phi\colon J_\phi\to\SS^2$ is the unit normal vector. The total variation norm of $D\phi$ is given by
\begin{align}
|D\phi|(U)=\Haus^2(J_\phi).\label{eq:Dphi_H2}    
\end{align}
By definition, functions $\phi\in SBV(U;\{0,1\})$ are characteristic functions of some $E\subset U$ of finite perimeter, i.e., $\phi=\charac_E\colon U\to\RR$ with
$$
\charac_E(x):=\begin{cases}
    1, & x\in E\\
    0, & \text{otherwise.}
\end{cases}
$$ 
In particular, $J_\phi$ coincides with the reduced boundary of $E$ up to a $\Haus^2$-null set, $\nu_\phi$ points in the interior of $E$, and  $|D\phi|(U)$ is the perimeter of $E$.  

For the convenience of the reader, we recall the compactness theorem for piecewise constant functions which follows from \cite[Thm.~3.23, Thm.~4.25]{AmFuPa:00}.
\begin{theorem}[Compactness of piecewise constant SBV-functions]\label{thm:SBVpwcpt} 
Let $U\subset\RR^3$ be an \linebreak open, bounded Lipschitz domain. Let $(\phi_n)_{n\in\NN}\subset SBV(U)$ be piecewise constant functions, satisfying
$$
\sup_{n\in\NN}\left(\|\phi_n\|_{L^\infty(U)}+\Haus^2(J_{\phi_n})\right)<\infty.
$$
Then there exists a piecewise constant function $\phi\in SBV(U)$ such that after passing to a subsequence, we have $\phi_n\to \phi$ in $L^1(U)$ \KB and \EEE $D\phi_n \wto^* D\phi$ in $\Radon(U;\RR^3)$ as $n\to\infty$.
\end{theorem}

\subsection{Oriented curvature varifolds}\label{sec:ocv}
We briefly introduce the relevant definitions for varifolds, restricting to two-varifolds in the open set $U\subset \RR^3$. Let $G_{2,3}$ denote the Grassmannian, i.e., the set of all two-dimensional linear subspaces of $\RR^3$, which we describe by their orthogonal projection matrices $P\in\RR^{3\times 3}$. We identify the oriented Grassmannian with the two-sphere $\mathbb{S}^2$ by representing an oriented two-dimensional subspace by its unit normal.

Following \cite{Hutchinson:86}, an {\em oriented two-varifold} in $U$ is a (nonnegative) Radon measure 
$$
V\in\Radon(U\times\SS^2).
$$
The mass of $V$ is the Radon measure $\mu_V\in\Radon(U)$ given by
$$
\mu_V(B)=V(B\times\SS^2) \quad \text{for all Borel sets} \quad B\subset U.
$$
By Riesz' Representation Theorem, e.g., \cite{Simon:83}, $V$ is defined through its action on continuous functions with compact support, given by
$$
\langle V,u\rangle=\int_{U\times\SS^2}u(x,\nu)\dint V(x,\nu)\quad\text{for all}\quad u\in C^0_c(U\times\SS^2).
$$
Pushforward of $V$ by the covering map  $q\colon U\times \SS^2\to U\times G_{2,3}$, $q(x,\nu)=(x,\mathbb{I}_{3\times 3}-\nu\otimes\nu)$ gives the (unoriented) two-varifold $\qsharp V\in\Radon(U\times G_{2,3})$, namely,
$$
\langle \qsharp V,v\rangle=\int_{U\times G_{2,3}}v(x,P)\dint (\qsharp V)(x,P)=
\int_{U\times\SS^2}v(q(x,\nu))\dint V(x,\nu)
$$ 
for all  $v\in C^0_c(U\times G_{2,3})$. 
Moreover, to every oriented two-varifold $V$ we can associate a two-current $T_V\in\mathcal{D}_2(U)$ by 
$$
\langle T_V,\omega\rangle =\int_{U\times\SS^2}\langle\star\nu,\omega(x)\rangle\dint V(x,\nu)
$$
for all $\omega\in C_c^\infty(U;\Lambda^2(\RR^3))$, the smooth, compactly supported two-forms in $U$. Here, $\star\nu\in \Lambda_2(\RR^3)$ stands for the simple two-vector associated to $\nu\in\SS^2$ through the Hodge star operator $\star$. The boundary of $T_V$ is the one-current $\partial T_V\in\mathcal{D}_1(U)$, which is given by 
$$
\langle \partial T_V,\eta\rangle=\langle T_V,\dint\eta\rangle
$$ 
for all one-forms $\eta\in C_c^\infty(U;\Lambda^1(\RR^3))$, where $\dint\eta$ denotes the exterior derivative of $\eta$.

An {\em oriented integral two-varifold} is a varifold $V\in \Radon(U\times \SS^2)$ given by
$$
\langle V,u\rangle=\int_{M}\left(u(x,\nu^M(x))\theta^+(x)+u(x,-\nu^M(x))\theta^-(x)\right)\dint \Haus^2(x)
$$
for all $u\in C^0_c(U\times\SS^2)$ and for which we will also write
\begin{align}\label{eq:rec_or_varif}
V(x,\nu)=(\Haus^2\llc M)(x)\otimes\left(\theta^+(x)\delta_{\nu^M(x)}(\nu)+\theta^-(x)\delta_{-\nu^M(x)}(\nu)\right).
\end{align}
Here, $M\subset U$ is a countably $\Haus^2$-rectifiable set, the {\em orientation} $\nu^M\in L^1_{\loc,\Haus^2}(M;\SS^2)$ selects one of the two unit normals 
to the approximate tangent plane $T_xM$ at $\Haus^2$-a.e.\ $x\in M$, and the corresponding {\em multiplicities} $\theta^\pm\in L^1_{\loc,\Haus^2}(M)$ are integer-valued, i.e., $\theta^\pm(x)\in\NN$ for $\Haus^2$-a.e.\ $x\in M$.  The class of oriented integral two-varifolds in $U$ is denoted by $IV_2^o(U)$.

The unoriented varifold associated to $V$ is the {\em integral two-varifold} given by 
$$
\langle \qsharp V,v\rangle=\int_{M}v(x,T_xM)(\theta^+(x)+\theta^-(x))\dint \Haus^2(x)
$$
for all  $v\in C^0_c(U\times G_{2,3})$. The class of (unoriented) integral two-varifolds in $U$ is denoted by $IV_2(U)$.
The current associated to $V$ is the {\em integral two-current} given by
$$
\langle T_V,\omega\rangle=\int_{M}\langle \star\nu^M(x),\omega(x)\rangle(\theta^+(x)-\theta^-(x))\dint \Haus^2(x)
$$
for all  $\omega\in C_c^\infty(U;\Lambda^2(\RR^3))$.

A {\it curvature two-varifold} (in the sense of \cite{Hutchinson:86}, i.e., without boundary measure \cite{Mantegazza:96}), denoted by $V\in CV_2(U)$, is an integral varifold $V\in IV_2(U)$ such that there exist {\em generalized curvature function}
$A^V=(A^V_{ijk})_{i,j,k=1}^3\in L^1_{\loc,V}(U\times G_{2,3};\RR^{3 \times 3 \times 3})$
satisfying
$$
\int_{U\times G_{2,3}}\sum_{j=1}^3\Big(P_{ij}\partial_j\varphi+\sum_{k=1}^3(\partial_{P_{jk}}\varphi)\,A_{ijk}^V+A_{jij}^V\,\varphi\Big)\dint V=0
$$
for all $\varphi\in C_c^1(U\times G_{2,3})$ and $1\leq i\leq 3$. 
An {\it oriented curvature two-varifold}, denoted by $V\in CV^o_2(U)$, is an oriented integral varifold $V\in IV_2^o(U)$ whose unoriented counterpart $\qsharp V$ is a curvature varifold, i.e.,
$$
CV_2^o(U)=\{V\in IV_2^o(U):\:\qsharp V\in CV_2(U)\}.
$$
We conclude with a compactness theorem for oriented curvature varifolds without boundary, still restricting to two-varifolds in $U\subset \RR^3$. The result is a combination of the compactness theorem \cite[Thm.\ 3.1]{Hutchinson:86} for oriented integral varifolds and \cite[Thm.\ 6.1]{Mantegazza:96} for curvature varifolds.
\begin{theorem}[Compactness of oriented curvature varifolds without boundary]\label{thm:CVocpt}\phantom{e}\linebreak
Let $p>1$ and $(V_n)_{n\in\NN}\subset CV_2^o(U)$ satisfy $\partial T_{V_n}=0$ for all $n\in\NN$ and
$$
\sup_{n\in\NN}\left(\mu_{V_n}(U)+\|A^{\qsharp V_n}\|^p_{L_{\qsharp V_n}^p(U\times G_{2,3})}\right)<\infty.
$$
Then there exists  $V\in CV_2^o(U)$ such that after passing to a subsequence, $V_{n}\wto^\ast V$ in $\Radon(U\times \SS^2)$ and $\qsharp V_n \wto^\ast \qsharp V$, $A^{\qsharp V_n}_{ijk}\qsharp V_n \wto^\ast A^{\qsharp V}_{ijk}\qsharp  V$ in $\Radon(U\times G_{2,3})$ for all $1\leq i,j,k\leq 3$.

\end{theorem}
Here, the $p$-th power of the $L^p$-norm of the curvature function of $V\in CV_2^o(U)$ is 
$$
\|A^{\qsharp V}\|^p_{L_{\qsharp V}^p(U\times G_{2,3})}=\int_{U\times G_{2,3}} |A^{\qsharp V}(x,P)|^p\,\dint(\qsharp V)(x,P)
$$
with Frobenius norm $|A^{\qsharp V}|=\sqrt{\sum_{i,j,k=1}^3(A^{\qsharp V}_{ijk})^2}$.
\KB In particular, if $V\in CV_2^o(U)$ corresponds to a smoothly embedded closed oriented surface $M\subset U$ and has multiplicities $\theta^++\theta^-\equiv 1$, cf.\ \eqref{eq:rec_or_varif}, then $\mu_V(U)=\int_M(\theta^++\theta^-)\dint\Haus^2=\Haus^2(M)$
and the Frobenius norms of the curvature function $A=A^{\qsharp V}$ of $M$ and of its second fundamental form $I\!I$ are related by $|A|^2=2|I\!I|^2$. 
\EEE 

\section{Model}\label{sec:model}
\subsection{States} 
Let $\Omega\subset\RR^3$ be an open, bounded Lipschitz domain, which describes the  reference configuration of an elastic body. 
The state of the body is characterized by the {\em deformation} $y$,
the {\em phase field} (or {\em phase indicator}) $\phi$, and the oriented curvature varifold
$V$ corresponding to the \emph{phase interfaces.}
The deformation is a homeomorphism
$$
y\colon\Omega\to y(\Omega)\subset\RR^3,
$$
mapping points in $X\in\Omega$ to points $x=y(X)\in y(\Omega)$.
We describe the interfaces in the Eulerian setting, that is, both $\phi$ and $V$ are defined on the current configuration 
$y(\Omega)$ of the body, namely
$$
\phi\in SBV(y(\Omega);\{0,1\})\qquad\text{and}\qquad V\in CV_2^o(y(\Omega)).
$$
Since $y$ is a homeomorphism, $y(\Omega)\subset \RR^3$ is an open set.

A crucial ingredient of the model is that we introduce a coupling relating the phase $\phi$ and the varifold $V$ in the following sense.
    \begin{definition}[Coupling]\label{def:coupling}
        Let $U\subset \RR^3$ be open and consider the linear map 
    	$$Q\colon C^0_c(U;\RR^3)\to C^0_c(U\times\SS^2), \quad (QY)(x,\nu):=Y(x)\cdot\nu.$$ 
    	An oriented varifold $V\in\Radon(U\times\SS^2)$ and a phase field $\phi\in SBV(U;\{0,1\})$ are called \emph{coupled in $U$} if
    	$D\phi=Q'V$, i.e., if for all $Y\in C^0_c(U;\RR^3)$ we have
     \begin{align}\label{eq:coupling}
         \langle D\phi,Y\rangle =\int_{J_\phi} Y\cdot\nu_\phi\,\dint \mathcal{H}^2  = \int_{U\times \SS^2}Y(x)\cdot\nu\,\dint V(x,\nu) = \langle V, QY\rangle = \langle Q'V,Y\rangle.
     \end{align} 
    \end{definition}
    Here, $Q'$ denotes the Banach space adjoint and we use the dualities $C^0_c(U;\RR^3)'=\mathcal{M}(U;\RR^3)$ and $C^0_c(U\times \mathbb{S}^2)'=\mathcal{M}(U\times \mathbb{S}^2)$, respectively.
    
	\begin{definition}[Admissible set]\label{def:admissible} We say that a triple $(y,\phi, V)$
 is \emph{admissible}, in short, 
 $$
 (y,\phi,V)\in\adm,
 $$
 if the following conditions are satisfied.
	\begin{enumerate}[(i)]
	\item\label{item:adm_y} $y\in W^{1,r}(\Omega;\RR^3)$ is a homeomorphism, $r>3$;
	
	\item $\phi\in SBV(y(\Omega);\{0,1\})$;

	\item $V\in CV_2^o(y(\Omega))$ 
 and has no boundary,
 i.e., 
 $\partial T_V=0$; 
	
	\item\label{item:coupling} $V$ is coupled to $\phi$ in $y(\Omega)$ in the sense of Definition \ref{def:coupling}. 
	\end{enumerate}
	\end{definition}
	
Condition \eqref{item:adm_y} is enforced by assuming that $y\in W^{1,r}(\Omega;\RR^3)$ satisfies $\det\nabla y>0$ a.e. in $\Omega$,  that it fulfills the Ciarlet-Ne\v{c}as condition \cite{CiaNec87ISCN}, \FFF i.e., \EEE
\begin{equation}\label{ciarlet-necas}
    \int_\Omega\det\nabla y(x)\,\dint x\le \mathcal{L}^3(y(\Omega)),
\end{equation}
 and  that the distortion $|\nabla y|^3/\det\nabla y\in L^{r-1}(\Omega)$. Namely, nonnegativity of the Jacobian determinant together with \eqref{ciarlet-necas} makes $y$  injective almost everywhere in $\Omega$. Controlling the distortion in $L^{r-1}(\Omega)$ ensures that $y$ is an open map; cf.~\cite[Thm.~3.24, p.~43]{HenKos14LMFD}. This together with almost everywhere injectivity implies that $y$ is homeomorphic in $\Omega$.
 
    \begin{lemma}[Properties of admissible triples]\label{lem:A_properties} Let $(y,\phi,V)\in \mathcal{A}$. Then we have
    \begin{enumerate}[(i)]
        \item\label{item:A_prop_J_vs_spt_mu} $J_\phi\subset \spt \mu_V$;
        \item\label{item:C2_rectif} $J_\phi$ and $\spt \mu_V$ are \emph{countably $\Haus^2$-rectifiable of class $2$,} i.e., up to a set of $\mathcal{H}^2$-measure zero, they can be covered by a countable union of embedded, two-dimensional $C^2$-submanifolds of $\RR^3$.
    \end{enumerate}
    \end{lemma}

    \begin{proof} 
        If we take a test vector field $Y\in C^0_c(y(\Omega);\RR^3)$ with $Y\equiv 0$ on $\spt\mu_V$, the coupling \eqref{eq:coupling} implies
        \begin{align}
            \int_{J_\phi} Y\cdot\nu_\phi\,\dint\mathcal{H}^2 
            = \int_{y(\Omega)\times \SS^2}\underbrace{Y(x)\cdot\nu}_{\leq |Y(x)|}\dint V(x,\nu)\leq \int_{y(\Omega)} |Y|\dint \mu_V =0,
         \end{align}
         and \eqref{item:A_prop_J_vs_spt_mu} follows. For \eqref{item:C2_rectif}, note that $A^{\qsharp V}\in L^1_{\mathrm{loc},\qsharp V}(y(\Omega)\times G_{2,3})$ implies that $\mu_V$ has locally bounded first variation with generalized mean curvature in $L^1_{\mathrm{loc},\mu_V}(y(\Omega))$, and consequently the statement follows from \cite[Theorem 1]{Menne}.
    \end{proof}

	\begin{remark}
		The coupling \eqref{item:coupling}, i.e., $D\phi=Q'V$ in $y(\Omega)$, does not imply that the multiplicity of $V$ must be one, in particular, it does not imply
  $V(x,\nu)=(\Haus^2\llc J_\phi)(x)\otimes\delta_{\nu_\phi(x)}(\nu)$, \KB cf.\ \eqref{eq:rec_or_varif}. \EEE Moreover, there may in general be multiple different varifolds coupled with a fixed phase $\phi$.
		The reason for this is that $Q$ is not surjective: 
		If $u\in C^0_c(y(\Omega)\times\SS^2)$ satisfies $u(x_0,\pm \nu)=1$ for some $x_0\in y(\Omega), \nu\in \mathbb{S}^2$, there is no representation $u=QY$ for $Y\in C^0_c(y(\Omega);\RR^3)$.
	\end{remark}

\subsection{Energies}\label{subsec:energies}
Equilibrium configurations of the body are admissible states $(y,\phi,V)\in\adm$ that minimize the {\em energy}
	\begin{equation}\label{eq:E}
		E(y,\phi,V):=\Ebulk(y,\phi)+\Einterface(y,V),
	\end{equation}
which consists of the {\em bulk energy}
$$
\Ebulk(y,\phi):=\int_\Omega\Big((1-(\phi\circ y)(X))\,W_0(\nabla y(X))+(\phi\circ y)(X)\,W_1(\nabla y(X))\Big)\dint X
$$
and the {\em interface energy}
\begin{align}
    \Einterface(y,V):=\int_{y(\Omega)\times G_{2,3}} \Psi(A^{\qsharp V}(x,P))\dint(\qsharp V)(x,P).
\end{align}


	Here,
$W_i\colon\RR^{3\times 3}\to \RR$ is the \emph{elastic energy density} of phase $i$ ($i=0,1$), and $\Psi\colon \RR^{3\times 3\times 3}\to \RR$ is the \emph{interface energy density.} We assume that there exist $c_{\textrm{bulk}}>0$ 
and $s>0$ such that
        \begin{align}\label{eq:W_coercive}
            W_i(F) \begin{cases}\geq c_{\textrm{bulk}} \left(|F|^{r} + \left(\displaystyle\frac{|F|^3}{\det F}\right)^{r-1} +(\det F)^{-s}\right) & \text{ if } \det F>0\\[2mm]
            =+\infty  &\text{ if } \det F\le 0,
            \end{cases}
        \end{align}
\KB where $r>3$ is the same as in Definition \ref{def:admissible}. \EEE         
          Moreover, we assume that $W_i$ is polyconvex \cite{Ball:77}, i.e., that there exists a convex function $h_i:\RR^{19}\to\RR$ such that $W_i(F)=h_i(F,
        \Cof F,\det F)$ for all $F\in\RR^{3\times 3}$ with $\det F>0$ and $i=0,1$. It is easily seen that $F\mapsto (|F|^3/\det F)^{r-1}$  is polyconvex in the set of matrices with positive determinants if $r>3$. Therefore the right-hand side of \eqref{eq:W_coercive} can serve as an example of a polyconvex stored energy density.

In the interfacial energy, we assume that $\Psi\colon \RR^{3\times 3\times 3}\to \RR$ is a convex function satisfying
\begin{align}\label{eq:Psi_coercivity}
    \Psi(A)  \geq c_{\mathrm{int}}(1+ |A|^p)\quad \text{ for all } A\in \RR^{3\times 3\times 3}
\end{align}
for some $p>1$ and $c_{\mathrm{int}}>0$. Integrating \eqref{eq:Psi_coercivity}, we find that the interface energy controls both the curvature and the mass of the varifold, since
\begin{align}\label{eq:Psi_control}
    \Einterface(y,V) \geq c_\mathrm{int}\Big(\mu_V(y(\Omega))+\int_{y(\Omega)\times G_{2,3}}|A^{\qsharp V}|^p \dint (\qsharp V)\Big).
\end{align}

\section{Existence of equilibrium states}\label{sec:existence}
	
	Let $E$ be given by \eqref{eq:E} and $\adm$ as in Definition \ref{def:admissible}. Then, 
	the following result holds.
	\begin{theorem}[Existence]\label{thm:main} There exists a minimizer of $E$ on $\adm$.
	\end{theorem}
	
	\begin{proof}
		We observe that $(y,\phi,V)=(\text{id},1,0)\in\adm$ and 
		$E(\text{id},1,0)=\int_\Omega W_1(\mathbb{I}_{3\times 3})\,\dint X<\infty$,
		which is a consequence of $W_1(\mathbb{I}_{3\times 3})<\infty$ and $|\Omega|<\infty$. In particular, $\inf_{\adm}E<\infty$.
		
		Let $(y_n,\phi_n,V_n)_{n\in\NN}\subset\adm$ be a minimizing sequence for $E$. Without loss of generality, we may assume $\int_\Omega y_n\,\dint X =0$ and
		\begin{align}
			E(y_n, \phi_n, V_n)\leq K\quad \text{ for all }n\in \NN.
		\end{align} 
        In particular, by \eqref{eq:W_coercive}, we have $\det \nabla y_n>0$ a.e.\ in $\Omega$ and $(y_n)_{n\in \NN}\subset W^{1,r}(\Omega;\RR^3)$ is bounded. Thus, after passing to a subsequence, we may assume $y_n\rightharpoonup y$ in $W^{1,r}(\Omega;\RR^3)$ and also 
		\begin{align}\label{eq:yn_unif}
			y_n\to y \quad \text{ in } C^0(\bar{\Omega};\RR^3) \text{ as }n\to\infty.
		\end{align}
  Moreover, the weak convergence of $(y_n)_{n\in\NN}$,  the sequential weak continuity of $y\mapsto\det\nabla y\colon W^{1,r}(\Omega;\RR^3)\to L^{r/3}(\Omega)$, and of  $y\mapsto\Cof\nabla y\colon W^{1,r}(\Omega;\RR^3)\to L^{r/2}(\Omega;\RR^{3\times 3})$ yield
  \begin{align}
  &\nabla y_n\rightharpoonup\nabla y  \quad \text{ in }L^r(\Omega;\RR^{3\times 3})  \text{ as }n\to\infty,\label{eq:nabla_conv}\\
			&\det \nabla y_n\rightharpoonup \det\nabla y \quad \text{ in } L^{r/3}(\Omega) \text{ as }n\to\infty,\label{eq:det_conv}\\
   &
			\Cof \nabla y_n\rightharpoonup \Cof\nabla y \quad \text{ in } L^{r/2}(\Omega;\RR^{3\times 3}) \text{ as }n\to\infty.\label{eq:cof_conv}
  \end{align}
      The convergence in \eqref{eq:yn_unif} allows us to pass  to the limit in the right-hand side of \eqref{ciarlet-necas} while the left-hand side passes to the limit due to the sequential weak continuity of $y\mapsto\det\nabla y$: $W^{1,r}(\Omega;\RR^3)\to L^{r/3}(\Omega)$.   
      Moreover, \eqref{eq:nabla_conv} and \eqref{eq:det_conv}, together with polyconvexity of $F\mapsto (|F|^3/\det F)^{r-1}$, imply that 
      $$\liminf_{n\to\infty}\int_\Omega \frac{|\nabla y_n|^{3(r-1)}}{(\det\nabla y_n)^{r-1}}\, {\rm d}X\ge \int_\Omega \frac{|\nabla y|^{3(r-1)}}{(\det\nabla y)^{r-1}}\, {\rm d}X, $$
      i.e., the distortion of the limit deformation also belongs to 
      $L^{r-1}(\Omega)$, in particular, it implies that 
      $y$ is also a homeomorphism.

Given a strictly decreasing zero-sequence $(\veps_\ell)_{\ell\in\NN}\subset (0,1)$,
let $(U^\ell)_{\ell\in\NN}$ denote a sequence of open, bounded Lipschitz domains such that
		\begin{align}
			 \{x\in y(\Omega):\: \mathrm{dist}(x, \partial y(\Omega))>\veps_\ell \} \subset U^\ell \subset\subset y(\Omega)
		\end{align}
and which, upon extracting a subsequence, is increasing, i.e., $U^{\ell_1}\subset U^{\ell_2}$ whenever $\ell_1<\ell_2$.
By \eqref{eq:yn_unif}, one easily shows that $U^\ell$ for $\ell\in\NN$ fixed is contained in the image set $y_n(\Omega)$, namely,
		\begin{align}\label{eq:Ueps_in_y_n}
			U^\ell \subset y_n(\Omega) \quad\text{whenever $n\geq n(\ell)\in\NN$ is large enough.}
		\end{align}
      
		{\bf Compactness:} First, we examine the sequence of varifolds. For $\ell\in\NN$ and $n\geq n(\ell)$, consider the restriction
		\begin{align}
			V_n^{\ell} := V_n \llc (U^{\ell}\times \SS^2).
		\end{align}
		By testing the definitions of curvature varifolds and current boundaries with test functions supported in $U^\ell\times \SS^2$ and $U^\ell$, one verifies that
		\begin{align}
			V_n^{\ell} &\in CV^{o}_2(U^{\ell}),\quad A^{\qsharp V_n^{\ell}} = A^{\qsharp V_n}\llc (U^{\ell}\times \SS^2), \quad \text{and}\quad \partial T_{V_n^\ell}=0 \text{ in }U^{\ell}.
		\end{align}
		The coercivity assumption \eqref{eq:Psi_coercivity} implies that
		\begin{align}
			c_{\textrm{int}}\Big(\mu_{V_n^\ell}(U^\ell)+\int_{U^{\ell}\times G_{2,3}}|A^{\qsharp V_n^{\ell}}|^p\,\mathrm{d}(\qsharp V_n^{\ell}) \Big)\leq \Einterface(y_n,V_n) \leq K.
		\end{align}
		After passing to a subsequence, it thus follows from Theorem \ref{thm:CVocpt}
  that there exist $V^{\ell}\in CV^{o}_2(U^{\ell})$ such that $V_n^{\ell}\wto^\ast V^{\ell}$ in $\mathcal{M}(U^{\ell}\times \SS^2)$ and $\qsharp V_n^\ell \wto^* \qsharp V^\ell$, $A_{ijk}^{\qsharp V_n^\ell} \qsharp V_n^\ell \wto^* A_{ijk}^{\qsharp V^\ell} \qsharp V^\ell$ in $\Radon(U^\ell \times G_{2,3})$ for $1\leq i,j,k\leq 3$. In particular, it follows $\partial T_{V^\ell}=0$ in $U^{\ell}$. 
  
		Similarly, we find that 
        \begin{align}
            \phi_n^{\ell}:= \phi_n\vert_{U^{\ell}}\in SBV(U^{\ell};\{0,1\}) \quad \text{ with }\quad D\phi_n^{\ell} = D\phi_n \llc U^{\ell}.
        \end{align}
       
By \eqref{eq:Dphi_H2} and since $J_{\phi_n^\ell}\subset \spt \mu_{V_n^\ell}$ by Lemma \ref{lem:A_properties}\eqref{item:A_prop_J_vs_spt_mu}, it follows that
		\begin{align}\label{eq:SBV_bounds_Ueps}
		\Haus^2(J_{\phi_n^\ell})=|D \phi_n^\ell|(U^\ell)\leq \mu_{V^{\ell}_n}(y_n(\Omega))\leq K.
		\end{align}  
Moreover, $\Vert \phi_n^{\ell}\Vert_{L^\infty(U^{\ell})}<\infty$ uniformly as well, because $\phi_n^{\ell}\in SBV(U^\ell;\{0,1\})$ and $U^\ell\subset y(\bar\Omega)$ is bounded.
Consequently, from Theorem \ref{thm:SBVpwcpt} 
    it follows that after passing to a subsequence, we have $\phi_n^{\ell}\to \phi^{\ell}$ in $L^1(U^{\ell})$ as $n\to\infty$ with also $\phi^{\ell}\in \{0,1\}$ a.e.\ and $D\phi_n^{\ell} \wto^\ast D\phi^{\ell}$ in $\mathcal{M}(U^{\ell};\RR^3)$.
		
     It is not difficult to see that the above limits are local, i.e., if $\ell<\ell'$, then we have
		\begin{align}
			V^{\ell} = V^{\ell'}\llc (U^{\ell}\times \SS^2),\quad \phi^{\ell}&= \phi^{\ell'}\vert_{U^{\ell}}. \label{eq:well_def}
		\end{align}		
     Now, \UUU choosing  an appropriate diagonal sequence, we thus may assume $V_n^\ell \wto^* V^\ell$ as $n\to\infty$ for all $\ell\in \NN$, and \EEE     
  we obtain a limit varifold $V\in CV^{o}_2(y(\Omega))$ by setting
		\begin{align}
			\langle V, u\rangle = \langle V^{\ell},u\rangle = \lim_{n\to\infty} \langle V_n^\ell, u\rangle 
		\end{align}
        for any $u\in C^0_c(y(\Omega)\times \SS^2)$ and any $\ell\in\NN$ such that $\spt u\subset U^{\ell}$.
		Similarly, we define $\phi \in L^1(y(\Omega);\{0,1\})$ by the condition
  $$
  \phi\vert_{U^{\ell}} = \phi^{\ell}
  $$
  for all $\ell\in\NN$.
  By \eqref{eq:well_def} above, $V$ and $\phi$ are well-defined. Moreover, we have that $\phi \in SBV(y(\Omega);\{0,1\})$ since by \eqref{eq:SBV_bounds_Ueps} and \eqref{eq:Dphi_H2} we have
        \begin{align}
            |D\phi| (y(\Omega)) \leq \sup_{\ell\in\NN}\liminf_{n\to\infty} |D \phi^\ell_n|(U^\ell)=\sup_{\ell\in\NN}\liminf_{n\to\infty} \Haus^2(J_{\phi_n^\ell})\leq K.
        \end{align}
		
		To see that $V$ is coupled to $\phi$, fix any $Y\in C^0_c(y(\Omega);\RR^3)$. Taking $\ell\in\NN$ sufficiently large, we have that $\spt Y\subset U^{\ell}$, and consequently
		\begin{align}
			\langle D\phi, Y\rangle &= \langle D\phi^{\ell}, Y\rangle = \lim_{n\to\infty}\langle D\phi_n^{\ell},Y\rangle = \lim_{n\to\infty}\langle D\phi_n, Y\rangle \\
			&= \lim_{n\to\infty} \langle V_n, QY\rangle  = \lim_{n\to\infty} \langle V_n^{\ell}, QY\rangle = \langle V^{\ell}, QY\rangle = \langle V, QY\rangle,
		\end{align}
		using that $V_n$ and $\phi_n$ are coupled in $y_n(\Omega)$.
		
		{\bf Lower semicontinuity:} 
		For the interface part of the energy, by convexity of $\Psi$ and \cite[Theorem 5.3.2]{Hutchinson:86}, for every $\ell\in\NN$ we have 
		\begin{align}\label{eq:lsc_varif_eps}
			\int_{U^{\ell}\times G_{2,3}} \Psi\big(A^{\qsharp V^{\ell}}\big)\mathrm{d}(\qsharp V^{\ell}) \leq \liminf_{n\to\infty}\int_{U^{\ell}\times G_{2,3}}\Psi\big(A^{\qsharp V_n^{\ell}}\big)\mathrm{d}(\qsharp V_n^{\ell}).
		\end{align}
         Sending $\ell\to\infty$, monotone convergence implies
		\begin{align}
			\Einterface(y,V)=\int_{y(\Omega)\times G_{2,3}}\Psi\big(A^{\qsharp V}\big)\mathrm{d} (\qsharp V) & \leq \liminf_{n\to\infty}\int_{y_n(\Omega)\times G_{2,3}}\Psi\big(A^{\qsharp V_n}\big)\mathrm{d}(\qsharp V_n)\\ 
   &=\liminf_{n\to\infty}\Einterface(y_n,V_n).
		\end{align}
		
		We now turn to the bulk term. From the above construction of $U^\ell$ it follows that
		\begin{align}
			&\phi_n \to \phi \quad\text{ in } L^1(U^\ell),
		\end{align} 
        for all $\ell\in\NN$. We now find that
        \begin{align}
            \int_{y_n(\Omega)\cap y(\Omega)} |\phi_n(x)-\phi(x)|\mathrm{d}x 
            &\leq \int_{U^\ell} |\phi_n(x)-\phi(x)| \mathrm{d}x + |y(\Omega)\setminus U^\ell|,
        \end{align}
        so that, by sending first $n\to\infty$ and then $\ell\to\infty$, by \eqref{eq:yn_unif} we conclude that 
        \begin{align}
            \lim_{n\to\infty}\Vert \phi_n-\phi\Vert_{L^1(y_n(\Omega)\cap y(\Omega))}= 0.
        \end{align}
        By the energy bound and the coercivity assumptions \eqref{eq:W_coercive}, we find that the $y_n$ have uniformly $L^{r-1}$-bounded distortion, $r-1>2$, and consequently the assumptions of \cite[Lemma 5.3]{GKMS19} are satisfied. This yields that 
        \begin{align}\label{eq:phi_circ_y}
            \phi_n \circ y_n  \to \phi \circ y \quad \text{ in }L^1(\Omega).
        \end{align}        
       \MK  The last limit passage, polyconvexity of the bulk energy densities, weak convergence of minors \eqref{eq:nabla_conv}, \eqref{eq:det_conv}, \eqref{eq:cof_conv}, and the lower semicontinuity result of Eisen \cite{Eisen} imply  that 
		\begin{align}
			\Ebulk(y,\phi)\leq \liminf_{n\to\infty} \Ebulk(y_n,\phi_n).
		\end{align}
		Consequently, a minimizer of $E$ in $\mathcal{A}$ exists. 
	\end{proof}
 \EEE
\begin{remark}[Multiple phases] \label{rem:multi}
 \MK  Although the model above deals only with two  phases, an extension to a general  multiphase material is possible in a similar way as in \cite{Baldo90}, or \cite{grandi2020,Silhavy-2011}. \UUU More precisely, one can describe the case of $m\in {\mathbb N}$ distinct phases by redefining 
 $$ 
 E(y,\phi,V) = \sum_{i=1}^m \left(\int_\Omega (\phi_i \circ y)W_i(\nabla y)\, {\rm d} X + c_i\,\Einterface(y,V_i)\right).
 $$
\UUU
Here, the phase descriptor $\phi=(\phi_1,\dots, \phi_m)$ takes values in $\{0,1\}^m$\KB , i.e., $\phi\colon y(\Omega)\to \{0,1\}^m$, \UUU and the components $\phi_i$ \KB $(i=1,\dots,m)$ \UUU describe the local proportion of the different phases. In particular, $\phi$ is
 constrained to the set of pure phases $\{\phi=(\phi_1,\cdots, \phi_m) \in \{0,1\}^m\: : \: \phi_1+\dots+\phi_m=1\}$. Correspondingly, the vector of varifolds $V=(V_1,\cdots,V_m)$ collects 
 $m$ varifolds, 
 \FFF such that $V_i$ is coupled to $\phi_i$ in $y(\Omega)$, $i=1,\dots,m$. \UUU The energy densities $W_i$ are all assumed to be coercive as in \eqref{eq:W_coercive} and the constants $c_i$ are assumed to be positive. Indeed, the latter positivity turns out to be necessary for lower semicontinuity, 
 \MK see \cite{AmbBra90a} for details.
   \EEE
\end{remark}

\section{Topology optimization}\label{sec:topology}

In this section, we build on the above theory and tackle a
problem in topology optimization \cite{Allaire,Bendsoe}. With respect to more classical settings, the novelty is twofold here. Firstly, in line with this note's general approach, the problem's description is fully Eulerian, which naturally corresponds to the large-deformation setting. Secondly, the curvature of
the material interface
is taken into account. The penalization of the
curvature of the boundary of the body, in addition to its surface area, fits well into applicative situations where sharp edges should
avoided. Indeed, in many mechanical applications sharp edges, especially reentrant edges, may be subject to strong stresses and are often the onset of plasticization and damage.

Given a deformation $y\colon \Omega\to \RR^3$ \KB of an open, bounded Lipschitz domain $\Omega\subset\RR^3$, \EEE we interpret  the \KB image \EEE set $y(\Omega) $
as an a priori unknown {\it design domain}. This has to be understood as a container, to be partially occupied by an elastic
solid, whose actual shape is the subject of the optimization procedure.
We reinterpret the phase indicator $\phi\colon y(\Omega)\to \{0,1\}$ as a descriptor of the optimal shape. More precisely, the deformed state of the body to be identified corresponds to the subset of $y(\Omega)$ where $\phi\KB\equiv\EEE 1$. On the contrary, the subset of $y(\Omega)$ where $\phi\KB\equiv\EEE 0$ is interpreted as the deformed state of a  very
compliant {\it Ersatz} material, which is still assumed to be elastic. As customary in topology optimization, in order to avoid trivial solutions, we prescribe the 
\MAK total mass by imposing 
\begin{equation}
    \label{eq:volume}
    \int_{\Omega}\phi\circ y\, {\rm d} X  = \eta\, {\mathcal L}^3(\Omega)
\end{equation} \EEE
for a fixed parameter $\eta\in(0,1)$.

As the deformation $y$ is a priori unknown, for mathematical convenience, the scalar field $\phi$ is defined from here on on the whole space $\RR^3$ without changing notation. We will refer to such a field $\phi:\RR^3 \to \{0,1\}$ as {\it Eulerian material distribution}
in the following. 

Given an Eulerian material distribution $\phi$, we start by
solving the equilibrium problem with some appropriate boundary conditions in the referential configuration. More precisely, we let
$\partial \Omega $ be decomposed into $\Gamma_{\rm D},\, \Gamma_{\rm
  N}\subset \partial \Omega$, which are assumed to
be open (in the topology
of $\partial \Omega$) with $\Gamma_{\rm D}\cap\Gamma_{\rm N}
=\emptyset$, $\overline{\Gamma}_{\rm D}\cup \overline{\Gamma}_{\rm N}
= \partial \Omega $ (closure taken in the topology of $\partial
\Omega$), and ${\mathcal H}^2(\Gamma_{\rm D})>0$. The body is assumed to be
 clamped on $\Gamma_{\rm D}$ and the set of {\em admissible
deformations} reads
$${\mathcal Y}=\{y \in W^{1,r}(\Omega;{\mathbb R}^3) \::\: \text{$y$ is a
  homeomorphism}, \ y=\text{id on} \ \Gamma_{\rm D}\}.$$
In addition, a traction $g \in
L^1(\Gamma_{\rm N}; {\mathbb R}^3)$ is
exerted at the boundary part $\Gamma_{\rm N}$ and the material is subjected to a force with given force density $f \in L^1(\Omega; {\mathbb R}^3)$. Force and traction could also be assumed to be formulated in Eulerian coordinates, as well. 

For all $\eta\in (0,1)$ fixed, the set of equilibria related with $\phi$ 
is defined as
\begin{align}
    {\mathcal Y}(\phi)&=\argmin \bigg\{
E_{\rm bulk}(y,\phi) - \int_\Omega \KB(\EEE\phi\circ y\KB)\EEE f\cdot y \,  {\rm d} X-
\int_{\Gamma_{\rm N}} g \cdot y \, {\rm d} \Haus^2 \KB : \EEE \\\
&\qquad\qquad \qquad \qquad \text{$y\in {\mathcal Y}$ is such that} \  \eqref{eq:volume} \ \text{holds}\bigg\}. \label{eq:Yphi_def}
\end{align}

By following the arguments from the proof of Theorem \ref{thm:main}, one readily checks that $\mathcal{Y}(\phi)$ is not empty, provided $\phi$ is such that the constraint \eqref{eq:volume} is satisfied by some $y\in \mathcal{Y}$. 
\MAK To this aim, we will assume that the identity is admissible in \eqref{eq:volume}, i.e., $\int_\Omega \phi\, \dint X = \eta \mathcal{L}^3(\Omega)$. 
For such $\phi$, \EEE even if nonempty, $\mathcal{Y}(\phi)$ may not be a singleton, for equilibrium deformations could be not unique.  
\EEE

Our goal is to minimize the {\it compliance} 
$$ C(y,\phi)  = \int_\Omega (\phi\circ y)f\cdot y
\, {\rm d} X + \int_{\Gamma_{\rm N}} g \cdot y \, {\rm d} \Haus^2$$
which is a measure of the elastic energy stored by the deformed piece at equilibrium. 
In order to describe the curvature of the Eulerian interface, we augment the description of the material by \KB an oriented \EEE curvature varifold, which we relate to $\phi$ as in Section \ref{sec:model}. Recalling $\mathcal{A}$ from Definition \ref{def:admissible}, the topology optimization problem reads 
\MAK
\begin{align}
     &\min\Big\{C(y,\phi) +  \Einterface(y,V):~\phi\in L^\infty(\RR^3), \Vert \phi\Vert_{\infty}\leq 1, \int_\Omega\phi\, \dint X = \eta \mathcal{L}^3(\Omega) \\
     &\qquad\qquad\qquad \qquad\qquad\qquad\; y\in \mathcal{Y}(\phi), (y,\phi\vert_{y(\Omega)},V)\in\mathcal{A}\Big\}.\label{eq:to2}
\end{align}
\EEE

The main result of this section is the following. 

\begin{theorem}[Topology optimization]\label{thm4}  Problem \eqref{eq:to2} admits a solution. 
\end{theorem}

\begin{proof}

Let us first check that the infimum in \eqref{eq:to2} is not $\infty$. To this aim, we construct an admissible triplet in the domain of the functional. Let $\phi_0=1_H\in SBV(\RR^3;\{0,1\})$, where $H\subset \RR^3$ is a half-space with ${\mathcal L}^3(\Omega \cap H)= \eta {\mathcal L}^3(\Omega)$. Call $P_0 = \partial H\subset\RR^3$ and let 
    $N_0\KB\in\SS^2\EEE$ be the unit normal vector to $P_0$ pointing towards the interior of $H$. Moreover, let $V_0\KB\in CV_2^o(\RR^3)\EEE$ be as in \eqref{eq:rec_or_varif} with $M=P_0$, $\nu^M(x)=N_0$, $\theta^+(x)=1$, and $\theta^-(x)=0$ for $x\in M$. This entails in particular that $V_0$ and $\phi_0$ are coupled in $\RR^3$.  Now, $y_0=\mathrm{id}$ and $\phi_0$ satisfy the constraint \eqref{eq:volume}, so by the discussion after \eqref{eq:Yphi_def}, the set ${\mathcal Y}(\phi_0)$ is not empty. Let now $y_\ast \in \mathcal{Y}(\phi_0)$ be given. We readily check that $(y_\ast, \phi_0\vert_{y_\ast(\Omega)},V_0\llc (y_\ast(\Omega)\times \SS^2))\in\mathcal{A}$. Moreover, as $P_0$ being a plane implies $A^{\qsharp V_0} \equiv 0$, we have that 
    $$C(y_\ast,\phi_0\vert_{y_\ast(\Omega)}) + \Einterface(y_\ast,V_0\llc (y_\ast(\Omega)\times \SS^2)) = C(y_\ast,\phi_0\vert_{y_\ast(\Omega)}) + \Psi(0){\mathcal H}^2 (P_0 \cap y_\ast(\Omega))<\infty,$$
    since $y_\ast \in W^{1,r}(\Omega;\RR^3)$, $r>3$,  implies that $y_\ast(\Omega)$ is bounded.
    
    Let $(y_n,\phi_n,V_n)$ be an infimizing sequence
  for problem \eqref{eq:to2}.
  \MAK By comparing with the identity in \eqref{eq:Yphi_def}, we observe that \EEE
  $\mathcal{Y}(\phi_n)$ is bounded in $W^{1,r}(\Omega;\RR^3)$, independently of $n\in \NN$. \MAK We \EEE may argue as in the proof of Theorem \ref{thm:main} to conclude that, after passing to a not relabeled subsequence, there exist $\phi\in L^\infty(\RR^3)$, $y\in {\mathcal Y}$, and $V\in CV_2^o(y(\Omega))$ with $(y,\phi\vert_{y(\Omega)},V)\in \mathcal{A}$ such that in particular 
  \begin{align}
    &y_n \to y \quad \text{in}  \  C^0(\bar \Omega;{\mathbb R}^3). \label{eq:c1}\\
    &\phi_n \wto^* \phi \quad \text{in} \ L^\infty(\RR^3),  \ \ \phi_n \to \phi\quad \text{in} \ L^1_\mathrm{loc}(y(\Omega)),\label{eq:c2}\\
    &\phi_n\circ y_n \to \phi\circ y\quad \text{in} \ L^1(\Omega),\label{eq:c3}
  \end{align}
and, we have by lower semicontinuity and \eqref{eq:c2} that $\Vert\phi\Vert_\infty \leq \liminf_{n\to\infty}\Vert \phi_n\Vert_\infty\leq 1$, as well as
  \begin{align} 
    \Ebulk(y,\phi)&\leq \liminf_{n\to\infty} \Ebulk(y_n, \phi_n),\label{eq:c6}\\
    \Einterface(y, V) &\leq \liminf_{n\to\infty} \Einterface(y_n,V_n).\label{eq:Eint_lsc}
  \end{align}
Equation \eqref{eq:c1} and \eqref{eq:c3} also imply
  \begin{align}
      C(y, \phi) &= \lim_{n\to\infty} C(y_n, \phi_n).\label{eq:c7}
  \end{align}
As the \MAK mass \EEE constraint \eqref{eq:volume} passes to the limit under \MAK \eqref{eq:c3}, 
\EEE from \eqref{eq:c1}, \eqref{eq:c3}, and \eqref{eq:c6} we get that $y \in {\mathcal Y}(\phi)$. Hence, owing to inequality \eqref{eq:Eint_lsc} and the convergence \eqref{eq:c7} we conclude that $(y,\phi,V)$ solves the topology optimization problem~\eqref{eq:to2}.
\end{proof}

\begin{remark}[Worst-case-scenario compliance]
    Given the Eulerian material distribution $\phi$, the set ${\mathcal Y}(\phi)$ may contain more than one equilibrium, for uniqueness may genuinely fail \cite{Spadaro}. In order to tackle this indeterminacy, one could consider solving a topology optimization problem \eqref{eq:to2} where the compliance $C(y,\phi)$ is replaced by the {\it worst-case-scenario} compliance
    $$C_{\rm max}(\phi) = \max_{y\in {\mathcal Y}(\phi)} C(y,\phi).$$
    Note that $C_{\rm max}(\phi)$ can be proved to be well-defined, as soon as ${\mathcal Y}(\phi)$ is not empty. 
    
    To treat the worst-case-scenario-compliance case, however, one needs to require some stability of \KB the \EEE set ${\mathcal C}(\phi) \subset {\mathcal Y}(\phi)$ \KB given by \EEE those equilibria $y\in {\mathcal Y}(\phi)$ realizing the maximum, namely, such that $C(y,\phi)=C_{\rm max}(\phi)$. A condition which would ensure the validity of an existence result in the spirit of Theorem~\ref{thm4} would be 
    \begin{align}
    &\forall \phi_n ,\, \phi \in L^\infty(\Omega), \ \| \phi_n\|_\infty,\, \| \phi\|_\infty\leq 1, \ \phi_n \to \phi \ \text{in} \ L^1(\Omega), \ \forall y \in {\mathcal C}(\phi), \\
    &\qquad \qquad \exists\ y_n \in {\mathcal Y}(\phi_n):  \quad C(y,\phi)\leq \liminf_{n\to \infty } C(y_n,\phi_n).
    \end{align}
    The latter entails the existence of a recovery sequence for each equilibrium $y\in {\mathcal C}(\phi)$. In particular, it is trivially satisfied in case the set of equilibria ${\mathcal Y}(\phi)$ is a singleton, i.e., in case of uniqueness. Even in the case of nonuniqueness, the above condition holds if $C(y,\phi)$ takes the same value for all $y\in {\mathcal Y}(\phi)$. This is for instance the classical case of buckling of a rod under longitudinal compression. 
\end{remark}

\begin{remark}[A referential formulation] The fully Eulerian setting above can be computationally challenging. One could resort to a more classical referential setting by identifying the optimized body via $\varphi: \Omega \to \{0,1\}$ defined on the fixed reference configuration, while still retaining the penalization of the curvature of the referential boundary of the body, in addition to its
referential surface area. In this setting, the volume constraint \eqref{eq:volume} can be simplified \KB to \EEE $\| \varphi\|_1=\eta {\mathcal L}^3(\Omega)$ for some given $\eta \in (0,1)$. 
The set of equilibrium deformations related with $\varphi$
is  defined as  
\begin{align}
&{\mathcal Y}(\varphi)=\argmin_{y \in {\mathcal Y}}\left\{
 \int_\Omega\Big((1{-}\varphi)\,W_0(\nabla y)+\varphi\,W_1(\nabla y)\Big)\dint X
- \int_\Omega \varphi f\cdot y \,  {\rm d} X-
\int_{\Gamma_{\rm N}} g \cdot y \, {\rm d} \Haus^2\right\}
\end{align}
which can be readily checked to be not empty. By defining the {\it referential} compliance  as
$$C_{\rm ref}(y,\varphi) =  \int_\Omega \varphi f\cdot y
\, {\rm d} X + \int_{\Gamma_{\rm N}} g \cdot y \, {\rm d} \Haus^2,$$
the referential topology optimization problem reads 
\begin{align}
     &\min \bigg\{ C_{\rm ref}(y,\varphi)+ \int_{\Omega\times G_{2,3}} \Psi(A^{\qsharp V})\, {\rm d } (\qsharp V)
     \: : \: \varphi \in SBV(\Omega;\{0,1\}), \ \| \varphi\|_1=\eta {\mathcal L}^3(\Omega), \\
     &\qquad \qquad \qquad y \in {\mathcal Y}(\varphi), \ V \in CV_2^o(\Omega), \ \partial T_V=0, \, \varphi \text{ and }V\text{ are coupled in }\Omega\bigg\}.\label{eq:to}
\end{align}
Note that the varifold $V$ is now defined in the fixed set $\Omega \times \SS^2$, and we are using the same notation of Section \ref{sec:ocv}. By arguing along the lines above, one can prove that the referential topology optimization problem \eqref{eq:to} admits a solution.   
\end{remark}

	\section*{Acknowledgements}
 US and MK are partially funded by the Austrian Science Fund (FWF) and the Czech Science Foundation (GA\v CR) through project I\,5149/R\,21-06569K. US is supported by the FWF through projects F\,65  and I\,4354. US and FR are supported by the  FWF project P\,32788.

\end{document}